\documentclass[12pt]{amsart}

\usepackage{mystyle}

\begin{document}

\title{Boundary regularity for the free boundary in the one-phase problem}

\author[H. Chang-Lara]{H\'ector Chang-Lara}
\address{Department of Mathematics, Columbia University, New York, NY 10027}
\email{changlara@math.columbia.edu}

\author[O. Savin]{Ovidiu Savin}
\address{Department of Mathematics, Columbia University, New York, NY 10027}
\email{osavin@math.columbia.edu}

\begin{abstract}
We consider the  Bernoulli one-phase free boundary problem in a domain $\Omega$ and show that the free boundary $F$ is $C^{1,1/2}$ regular in a neighborhood of the fixed boundary $\partial \Omega$. We achieve this by relating the behavior of $F$ near $\partial \Omega$ to a Signorini-type obstacle problem.
\end{abstract}


\maketitle


\section{Introduction}\label{sec:intro}

The Bernoulli one-phase problem consists in finding  a nonnegative function $u$ which is fixed on the boundary $\partial \Omega$ of some given domain $\W\ss\R^n$, such that $u$ is harmonic in its positive set $\W^+ = \{u>0\} \cap \W$, and $u$ has a prescribed gradient over the \textit{free boundary} $F = \p \W^+ \cap \W$. Precisely, given $\Omega$ and two functions $g \ge 0$ and $Q>0$, we need to find $u$ which satisfies
\begin{align*}
\begin{cases}
\D u = 0 \text{ in } \W^+ = \{u>0\} \cap \W,\\
|Du| = Q(x) \text{ on } F = \p \W^+ \cap \W,\\
u = g \text{ on } \p\W.
\end{cases}
\end{align*}
In hydrodynamics these equations can be found in models of jets and cavities where the solution $u$ is the stream function for an incompressible and irrotational fluid \cite{MR679313}.

Solutions can be constructed either variationally as critical points of the associated energy functional (see \cite{MR618549})
$$J(u)=\int_\Omega |D u|^2 + Q^2\chi_{\{u>0\}} \, \, dx, $$
or by a viscosity solution approach using Perron's method \cite{MR1029856}. 

The local regularity theory for the free boundary $F$ at interior points of $\Omega$ is available for solutions $u$ which satisfy an additional nondegeneracy condition which requires for $u \sim dist(\cdot, F)$.  This was developed by Caffarelli in a series of papers in the 80's. The nondegeneracy condition is satisfied for example if either a) $u$ is a {\it minimizer} of the functional $J$ or b) $u$ is the {\it minimal supersolution}.

In this paper we address the regularity of the free boundary $F$ near a portion of the fixed boundary $\p \Omega$ where $u$ vanishes.

The situation is the following. We assume that $g=0$ over a portion of the boundary $Z\ss \p\W$, with $Z$ relatively open in the induced topology of $\p \Omega$. Assume for simplicity that $Z$ is locally a smooth hypersurface. We are interested in the behavior of the free boundary $F$ near $Z$, or in other words how $F$ separates from $Z$. Notice that $Z$ acts as an obstacle for the ``extension" $\bar F$ of the free boundary $F$ to the whole $\bar \Omega$ which is defined as $$\bar F = \p\W^+\cap \{u=0\}.$$

Moreover, it is not difficult to check that if we are in either of the situations a) or b) above then 
\begin{align*}
|Du| \geq Q(x) \text{ over } \L = \bar F \cap Z.
\end{align*}
This can be interpreted as a nondegeneracy condition for $u$ on the coincidence set $\L$, or equivalently as a stability condition for $F$. From the point of view of hydrodynamic models, the separation of $\bar F$ from $Z$ describes how the fluid detaches from a fixed boundary with slip condition.

\begin{figure}[t!]
\begin{center}
\includegraphics[width=14cm]{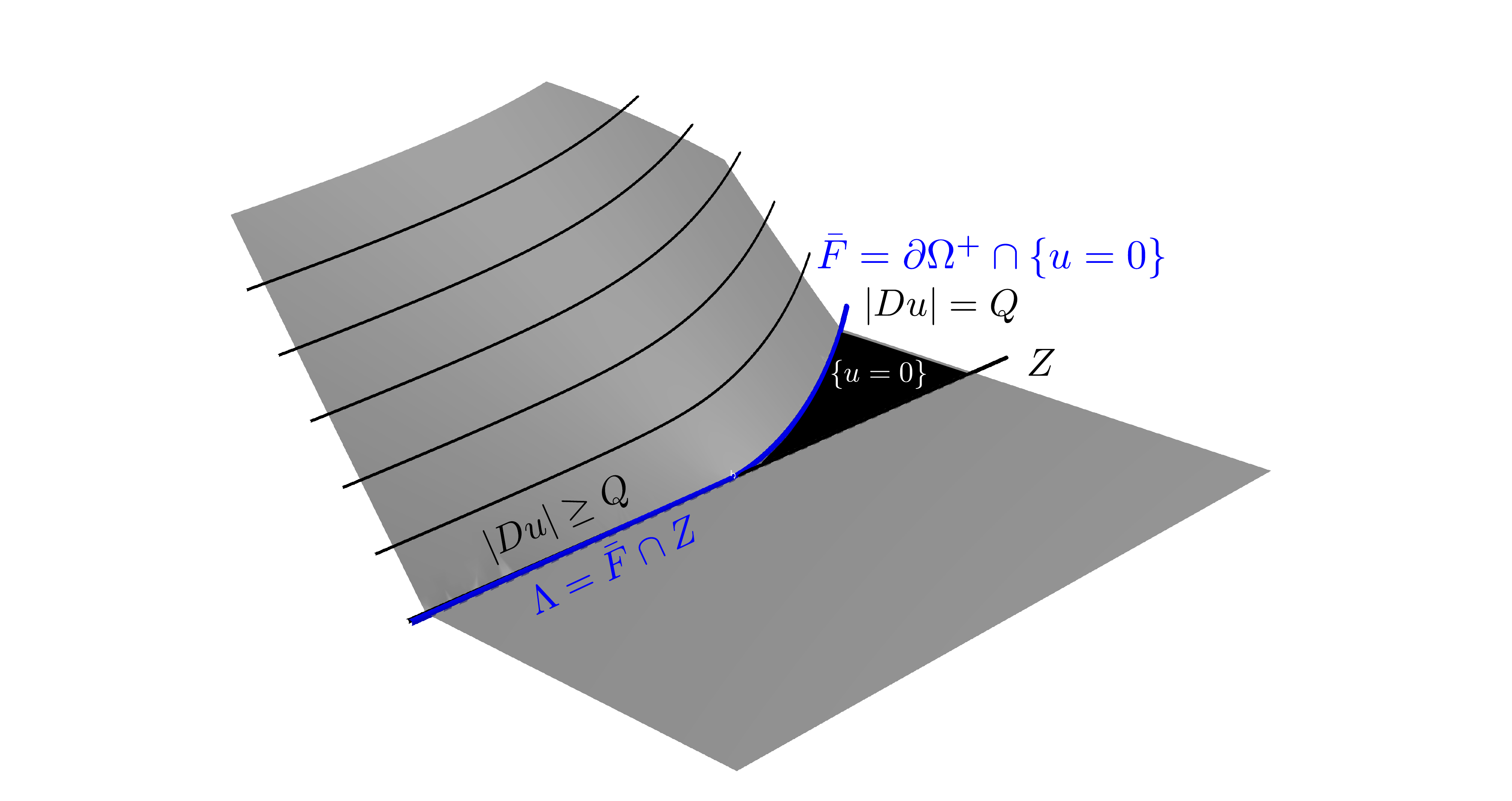}
\end{center}
\caption{Graph of the solution $u$.}
\label{fig:sigma_in_gamma}
\end{figure}

Our main result is the following.

\begin{theorem}\label{thm:main1}
Let $\W\ss\R^n$ be a domain with a $C^{1,\a}$ boundary portion $Z\ss\p\W$ for some $\a>1/2$, and let $Q\in C^{0,1}(\bar \W )$, $Q>0$. Let $u: \bar \Omega \to \R^+$ be a viscosity solution of
\begin{align*}
\begin{cases}
\D u = 0 \text{ in } \W^+ = \{u>0\}\cap \W,\\
u = 0 \text{ on } Z,\\
|Du| \geq Q(x) \text{ on } \bar F \cap Z,\\
|Du| = Q(x) \text{ on } F.
\end{cases}
\end{align*}
Then $\bar F$ is $C^{1,1/2}$ regular in a neighborhood of every $x_0 \in \L = \bar F\cap Z$.
\end{theorem}

Next we illustrate the main idea of Theorem \ref{thm:main1} by formally linearizing the one-phase problem near a point in $\L$. 

Assume for simplicity that $\W = B_1^+ = B_1\cap \{x_n>0\}$ and $Z = B_1' = B_1\cap \{x_n=0\}$, with $Q \equiv 1$ and say that $\bar F$ separates from $Z$ at the origin. As a first order approximation, we expect $$u = x_n + o(|x|).$$
Let $u = x_n - \e w$ in $\W^+$. Then the perturbation $w$ is harmonic, and moreover $w$ is nonnegative over $\bar F$. The free boundary condition over $F$ can be written as $|Du| = 1$ or, in terms of $w$, as
\begin{align*}
\p_n w = \frac{\e}{2}|Dw|^2 \text{ on } F.
\end{align*}
Additionally, $|Du| \geq 1$ on $\L$ means that
\[
\p_n w \leq 0 \text{ on } \L.
\]
As $\e\to 0$, we expect $\W^+\to B_1^+$, $\bar F \to B_1'$, and $w$ to solve
\begin{align}\label{SP}
\begin{cases}
\D w = 0 \text{ in } B_1^+,\\
w \geq 0 \text{ on } B_1',\\
\p_n w = 0 \text{ on } \{w>0\}\cap B_1',\\
\p_n w \leq 0 \text{ on } B_1'.
\end{cases}
\end{align}
These equations are known as the \textit{Signorini problem} or the \textit{thin obstacle problem}. Our main result states that $\bar F = \{x_n = \e w\}$ inherits the optimal regularity of the solution for the Signorini problem established by Athanasopoulos and Caffarelli in \cite{MR2120184}.

The regularity stated in our main result is optimal in terms of the regularity for $Q$. Consider in polar coordinates $u=(r\sin\theta-r^{3/2}\cos(3\theta/2))_+$ such that $F = \{\sin \theta = r^{1/2}\cos(3\theta/2)\}$ and over such set we get
\[
|Du|^2 = 1+9r/4-3\frac{\sin\theta\sin(\theta/2)}{\cos(3\theta/2)},
\]
From here we can extend $Q$ to a global Lipschitz function such that $u$ solves the corresponding one-phase problem on the upper half space.


\subsection{Previous results and overview of the paper}

For the one-phase problem, Alt and Caffarelli showed in \cite{MR618549} that $F$ is smooth outside of a set of $\mathcal H^{n-1}$ measure zero. Their proof is inspired by the regularity theory of minimal surfaces. The key estimate in \cite{MR618549} states that the free boundary is $C^{1,\a}$ regular provided a flatness hypothesis. A more general theory for two-phase problems was later developed by Caffarelli in \cite{Caffarelli1987a, MR973745, MR1029856} based on a viscosity solution approach.
 
Following the methods in \cite{Caffarelli1987a, MR973745, MR1029856}, several authors extended the results in different directions, for instance the case of variable coefficients with several types of regularity. At this point we would like to highlight one of the recent results due to De Silva, Ferrari, and Salsa \cite{MR3485136} as it will be relevant for our work. The main theorem in \cite{MR3485136} establishes that flat free boundaries are $C^{1,\alpha}$ regular in the case of divergence operators with H\"older continuous coefficients. The strategy is based in a compactness approach started by De Silva in \cite{MR2813524}.

The optimal regularity for the solution of the Signorini problem was first established by Athanasopoulos and Caffarelli in \cite{MR2120184}. The regularity of the free boundary around regular points was established by Athanasopoulos, Caffarelli and Salsa in \cite{MR2405165}. In this last reference the authors follow a blow-up procedure based on the monotonicity of the Almgren frequency formula which also plays an important role in our theorem.
 
We combine some of the recent strategies for the one-phase and the Signorini problem to prove our regularity result. In Section \ref{sec:alm_opt} we obtain the $C^{1,\b}$ regularity of $\p\W^+\cap B_1$ for every $\b\in(0,1/2)$ by following the compactness approach from \cite{MR2813524}. In Section \ref{sec:opt} we establish the monotonicity of an Almgren's type frequency formula in order to achieve the optimal $C^{1,1/2}$ regularity for $\p\W^+\cap B_1$. This section bears some similarities with work by Guill\'en \cite{MR2558329} and Garofalo, Smith Vega Garcia \cite{Garofalo2014} for the Signorini problem with variable coefficients.

\section{Preliminaries}\label{sec:prelim}

In this section we state the notion of solutions for the one-phase problem in the viscosity sense. A change of variables allows us to reformulate the problem over a convenient geometry. As a trade off we need consider operators with variable coefficients.

Let $a^{ij}$ symmetric and uniformly elliptic with respect to some fixed $\l>0$
\[
 \l|\xi|^2 \geq a^{ij}(x)\xi_i\xi_j \geq \l^{-1}|\xi|^2.
\]
We denote
\[
Lu = \p_i(a^{ij}(x)\p_j u) \qquad\text{ and }\qquad |D_a u| = \sqrt{a^{ij}(x)\p_i u\p_j u}.
\]
We say that $u$ is \textit{$L$-superharmonic ($L$-subharmonic or $L$-harmonic)} if $Lu \leq (\geq \text{ or } =) \ 0$ holds in the weak sense.

From now on we fix $Q$ continuous such that
\[
\text{$Q_{min}\leq Q(x) \leq Q_{max}$ for some fixed $0<Q_{min}\leq Q_{max}$}.
\]

Let $\varphi \in C(B_r(x_0))$ be nonnegative and $L$-superharmonic over $\W^+ = \{\varphi>0\} \cap B_r(x_0)$. We call $\varphi$ a \textit{strict comparison supersolution (subsolution) of the one-phase problem} if $\varphi \in C^1(\overline{\W^+})$ and
\[
|D_a\varphi| < (>) \ Q(x) \text{ in } \p\W^+\cap B_r(x_0).
\] 

Given $u,\varphi \in C(S)$, we say that \textit{$\varphi$ touches $u$ from above (below) at $x_0\in S$} if
\[
u(x_0)=\varphi(x_0) \text{ and } u \leq (\geq) \ \varphi \text{ in $S$}.
\]

Let $u \in C(\W)$ be nonnegative and $L$-subharmonic in $\W^+=\{u>0\}\cap \W$. We call $u$ a \textit{viscosity subsolution (supersolution) of the one-phase problem}
\begin{align}
\label{eq:vsol}
\begin{cases}
Lu = 0 \text{ in } \W^+ = \{u>0\}\cap \W,\\
|D_au| = Q(x) \text{ on } F = \p\W^+\cap \W,
\end{cases}
\end{align}
if there is no strict comparison supersolution (subsolution) that touches $u$ from above (below). We call $u$ is a \textit{viscosity solution} of \eqref{eq:vsol} if $u$ is simultaneously a subsolution and a supersolution.

Next we define the notion of viscosity solution up to the boundary. We assume for simplicity that $\Omega$ is a Lipschitz domain and $Z \ss \p \W$ is a relatively open set in $\p \W$ which is locally a $C^{1,\a}$ hypersurface. As in the introduction we denote by $\bar F$ the extension of $F$ to $\bar \Omega$,
$$\bar F = \p\W^+\cap \{u=0\}.$$
Following the terminology of the obstacle problem, we define the \textit{contact set} as
\[
\L = \bar F \cap Z.
\]
Finally we denote by $\p'\L$ the (thin) boundary of $\L$ relative to $Z$.

\begin{definition}\label{def:vis}
Let $u \in C(\bar \W)$ be nonnegative which vanishes on $Z$. We say that $u$ is a viscosity solution of the one-phase problem up to $Z$
\begin{align}
\label{eq:def}
\begin{cases}
Lu = 0 \text{ in } \W^+,\\
u = 0 \text{ on } Z,\\
|D_au| = Q(x) \text{ on } F,\\
|D_au| \geq Q(x) \text{ on } \L,
\end{cases}
\end{align}
if it is a viscosity solution of the one-phase problem on $\W$ and the last inequality is satisfied in the viscosity sense, i.e. $u$ cannot be touched from above at $x_0\in\L$ by a $C^1$ function $\varphi$ with $|D_a\varphi(x_0)| < Q(x_0)$.
\end{definition}

\begin{remark}\label{rmk:vis}
An equivalent definition can be given by extending $\W$ to some domain $U \supset \W$ such that $Z=\p \W \cap U$. Then $u \in C(\bar U)$ nonnegative which vanishes on $U \sm \W$ is a solution of \eqref{eq:def} if is a solution of the one-phase problem on $\W$ and a subsolution on $U$.
\end{remark}

\subsection{Existence}

Variational solutions for the one-phase problem can be constructed as minimizers of the following functional with $g \in H^1(\W)$ nonnegative
\[
J(u) = \int_\W  a^{ij}\p_iu\p_ju + Q^2\chi_{\{u>0\}} \, \, dx \qquad\text{ over }\qquad K = \{u\in H^1(\W): u-g \in H^1_0(\W)\}.
\]
In \cite{MR618549} it was shown that such minimizers of $J$ are viscosity solutions of the one-phase problem. We remarked in the introduction that they also satisfy Definition \ref{def:vis}.

\begin{lemma}\label{lem:perrons}
Let $g=0$ on $Z$. A minimizer of $J$ is a viscosity solution of \eqref{eq:def}.
\end{lemma}

Caffarelli developed in \cite{MR1029856} the Perron's method for viscosity solutions of a family of free boundary problems, including the one-phase problem. The idea is to construct a \textit{minimal viscosity solution} as the infimum over a family of \textit{admissible supersolutions} above a given \textit{subsolution minorant}. The main Theorem in \cite{MR1029856} states that the \textit{minimal viscosity solution} is a viscosity solution. We refer to \cite{MR1029856} for the precise definitions.

\begin{lemma}
Let $g=0$ on $Z$. The minimal viscosity solution above a subsolution minorant that vanishes over $Z$ is a viscosity solution of \eqref{eq:def}.
\end{lemma}

The proof of both lemmas can be achieved in similar ways by a contradiction argument. For instance, if there is a test function $\varphi$ touching $u$ from above at some $x_0 \in \L$ such that $|D_a\varphi(x_0)| < Q(x_0)$, then by applying a inward deformation of $\W^+$ as in \cite[Lemma 9]{MR1029856} we obtain an admissible supersolution smaller than $u$. The same deformation decreases the energy $J$.

From now on we assume that 
$$a^{ij},Q\in C^\a, \mbox{ and $Z$ is $C^{1,\a}$ regular},$$ 
and we focus our attention in a neighborhood of a point $x_0 \in \L$. After a domain deformation, we may reduce our analysis to the case
$$\W=B_1^+, \quad Z=B_1', \quad 0 \in \L.$$ 
In this case we will frequently consider $u$ to be defined over $B_1$ such that $u$ vanishes over $B_1\sm B_1^+$ and it is a subsolution of the one-phase problem over $B_1$, see Remark \ref{rmk:vis}.

\subsection{Lipschitz regularity and flatness of $u$}

Let $u \in C(B^+_1)$, nonnegative, harmonic in $\W^+ =\{u>0\}\cap B_1$, with $0\in\p\W^+$. If $\W^+$ satisfies either the interior or exterior ball condition at $0$, then a barrier argument shows that $u$ has a linear asymptotic behavior at $0$. From this observation we get to define the non-tangential gradient $Du(0)$ such that
\[
u(x) = (Du(0)\cdot x)_+ + o(|x|) \text{ as $x\to0$ non-tangentially in $\W^+$,}
\]
See \cite[Chapter 11]{MR2145284}. The same result can be reproduced for $L$ with $a^{ij}\in C^\a$ thanks to the Schauder estimates.

In the case of $u$ being a solution of \eqref{eq:def} with $\W=B^+_1$, $Z = B'_1$, we get that all points in $\L$ are regular from outside and it is then possible to construct barriers to bound $|Du|$ in terms of $\|u\|_{L^\infty(B^+_1)}$. This ultimately implies the Lipschitz regularity of the solution up to $\L$.

\begin{lemma}\label{lem:lipschitz_reg} Let $u$ a viscosity solution of \eqref{eq:def} with $\W=B^+_1$, $Z =B'_1$. Then 
\begin{align*}
\|u\|_{C^{0,1}\1B^+_{1/2}\2} \leq C\11+\|u\|_{L^\infty\1B^+_1\2}\2.
\end{align*}
\end{lemma}

On the other hand, the slope $|Du(0)|$ is bounded from below by $Q_{min}>0$. The asymptotic expansion and the Lipschitz regularity allows us deduce the following flatness result.

\begin{lemma}\label{lem:lin_behavior}
Let $u$ a viscosity solution of \eqref{eq:def} with $\W=B_1$, $Z = B'_1$ such that $0 \in \L$. Then, given $\e>0$ there exists $\d>0$ such that
\[
(|Du(0)|+\e)x_n \geq u \geq |Du(0)|(x_n-\e\d)_+ \text{ in } B_\d^+.
\]
\end{lemma}

\subsection{Interior regularity of flat free boundaries}

Finally we would like to recall one of the main results proved by Alt and Caffarelli in \cite{MR618549}: sufficiently flat free boundaries of the one-phase problem are $C^{1,\b}$. See also the recent results by De Silva, Ferrari, and Salsa for (two-phase) problems with divergence operators \cite{MR3485136}.

In the following we suppose $Q\in C^{0,1}$, and $a^{ij}\in C^\a$ such that for some $\e>0$,
\begin{align*}
a^{ij}(0) = \d^{ij}, \qquad Q(0)=1, \qquad \|a^{ij}-\d^{ij}\|_{C^{\a}(B_1)}+\|Q-1\|_{C^{0,1}(B_1)}\leq \e^2. 
\end{align*}
We assume $u \in C(B_1)$ to be a viscosity solution of
\begin{align*}
\begin{cases}
Lu = 0 \text{ in } \W^+ = \{u>0\}\cap B_1,\\
|D_au| = Q(x) \text{ on } F = \p\W^+\cap B_1,
\end{cases}
\end{align*}
such that
\[
B_1 \cap \{x_n>-\e\} \supseteq \W^+ \supseteq B_1 \cap \{x_n>\e\}.
\]

\begin{theorem}[DFS]\label{thm:classic}
For any $\b\in (0,1)$, there exists $\e_0\in(0,1)$ such that if $\e\in (0,\e_0)$ then $F$ gets parametrized by a $C^{1,\b}$ function in $B_{1/2}'$
\[
F\cap B_{1/2} = \{x_n = \e\bar u(x'): x' \in B_{1/2}'\},
\]
with the following estimate for some universal $C>0$,
\[
\|\bar u\|_{C^{1,\b}(B_{1/2}')} \leq C.
\]
\end{theorem}

\section{Almost optimal regularity}\label{sec:alm_opt}

In this section we show that $\bar F$ has almost optimal regularity in a neighborhood of $Z$. Precisely, we will show that if $x_0\in \L$, then $F$ is a $C^{1,\b}$ regular surface in a neighborhood of $x_0$ for any $\b\in (0,\min(1/2,\a))$, see Proposition \ref{thm:almost_opt}.

After a domain deformation and a dilation we assume as before that
$u$ is a viscosity solution of \eqref{eq:def} with $$\W=B^+_1,  \quad Z =B'_1, \quad 0 \in \L,$$ and that for $\a\in(0,1)$ the following smallness hypothesis for the coefficients hold
\begin{align}
\label{eq:hyp_calpha}
\tag{$C_{\e,\d}$} 
a^{ij}(0) = \d^{ij}, \qquad Q(0)=1, \qquad \|a^{ij}-\d^{ij}\|_{C^{\a}(B_1)}+\|Q-1\|_{C^{\a}(B_1)}\leq \delta \e. 
\end{align}
for some $\e\in(0,\e_0)$ and with $\d$ and $\e_0$ small, universal, to be made precise later.

The next Lemma will be used to show that $\{|Du|>Q\}\cap \L$ is open relative to $\{x_n=0\}$.

\begin{lemma}\label{lem:sigma_in_gamma}
Given $\eta>0$ there exists $\e_0>0$ such that if $\e\in(0,\e_0)$ then
\[
 u > (Q(0)+\eta)(x_n-\e)_+ \text{ in } B_1^+ \qquad \Rightarrow \qquad |D_au(0)|>Q(0).
\]
\end{lemma}

\begin{proof}
Let $\Omega_0$ be the domain above the parabola $P:=\{x_n = 8 \e |x'|^2\}$ that lies inside the cylinder $B'_{1/2} \times [0, 1/2]$, i.e.
$$\Omega_0:=\{1/2>x_n> 8 \e |x'|^2 \} \cap \{|x'| < 1/2 \}.$$

\begin{figure}[t!]
\begin{center}
\includegraphics[width=10cm]{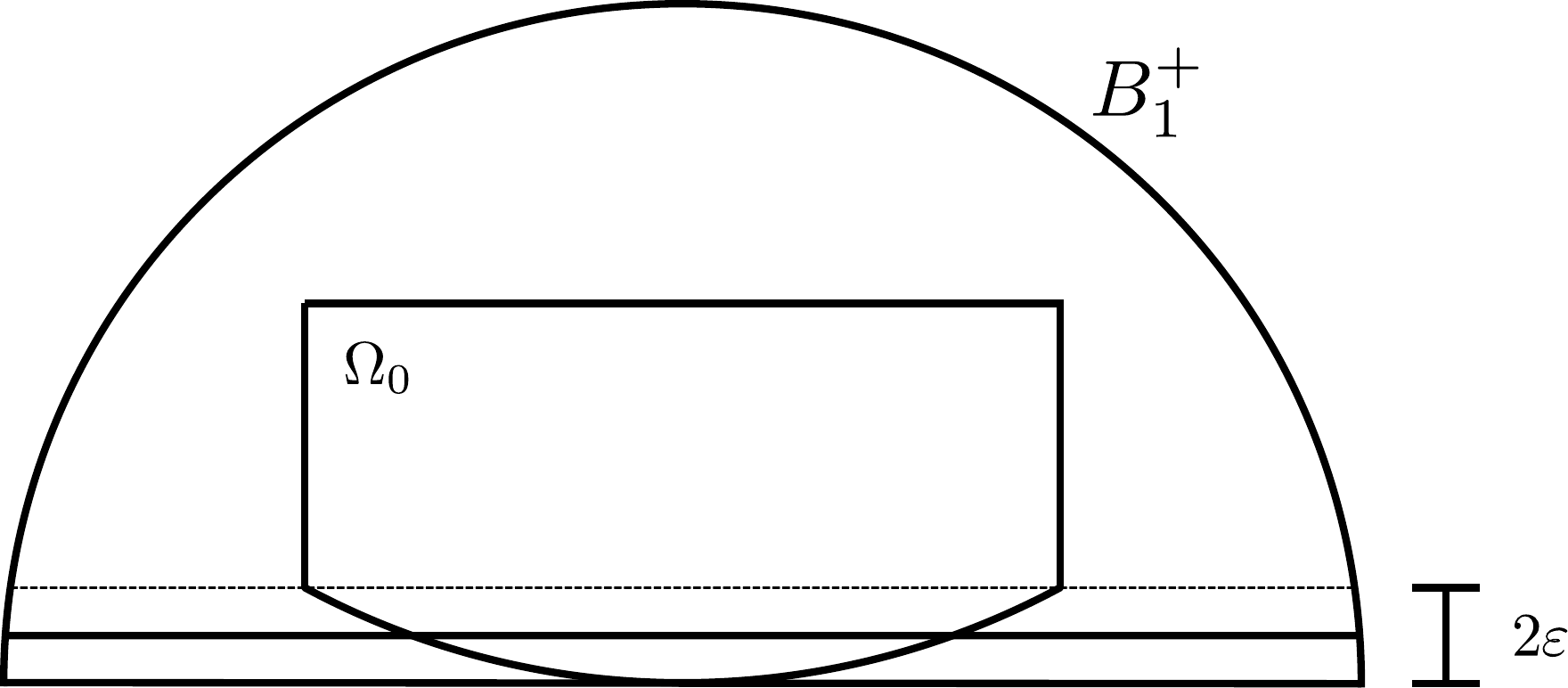}
\end{center}
\caption{Configuration for the proof of Lemma \ref{lem:sigma_in_gamma}}
\label{fig:sigma_in_gamma}
\end{figure}

 Define $\varphi_0$ in $\Omega_0$ as the solution to
\[
\begin{cases}
 L\varphi_0 = 0 \text{ in } \Omega_0,\\
 \varphi_0 = x_n-8 \e |x'|^2 \text{ on } \p \Omega _0.
\end{cases}
\]
 By \eqref{eq:hyp_calpha} we easily get that $\varphi_0$ is an $\e$-perturbation of $x_n$ and by the Schauder estimates up to the boundary we obtain that
\[
 ||D_a\varphi_0| - 1| \leq C\e  \text{ over } P \cap \{x_n \le \e\}.
\]

The hypothesis implies that $\psi_0:=(1+\eta/2) \varphi_0$ is below $u$ on $\p \Omega_0 \cap \{x_n> \e\}$, and the inequality above says that $|D_a \psi_0| > 1+\e > Q$ on the remaining part of the boundary, provided that $\e < c \eta$.

Let $\psi_t(x):=\psi_0(x-te_n)$ be the translation of $\psi_0$ by $t e_n$. Notice that the graph of $\psi_\e$ is below the graph of $u$. Then we slide the graph of $\psi_\e$ in the $-e_n$ direction till it coincides with $\psi_0$ (i.e. decrease $t$ from $\e$ to $0$). The graph of $\psi_t$ cannot touch the graph of $u$ neither on the free boundary, nor on the remaining part of the boundary of $\p \Omega_t \cap \{ x_n > \e\}$. In conclusion $u \ge \psi_0$ which gives the desired claim. 
\end{proof}

Thanks to Lemma \ref{lem:lin_behavior} any point in $\{|Du|>Q\} \cap \L$ satisfies the hypothesis of the previous lemma after a sufficiently large dilation (with $\eta$ depending on $|Du|-Q$). By applying the previous result centered at points in a sufficiently small neighborhood we conclude that $\{|Du|>Q\}\cap \L$ is open relative to $\{x_n=0\}$. To prove Theorem \ref{thm:main1} we can now focus on the case where $0 \in \L$ with $Du(0)=e_n$, i.e. when $0$ belongs to the thin boundary $\partial ' \L$. By invoking once again Lemma \ref{lem:lin_behavior}, we get that after a sufficiently large dilation we can start with a flatness hypothesis of the form
\begin{align}\label{eq:flat_hyp}
\tag{$F_\e$} x_n + \e  \geq u \geq (x_n -\e)_+ \text{ in $B_1^+$},
\end{align}
for some small $\e$.
Let us recall the Harnack inequality from \cite[Theorem 4.1]{MR3485136}.

\begin{lemma}\label{lem:dim_osc0}
Let $v$ be a viscosity \textbf{solution} of \eqref{eq:vsol} in $B_1$. There exist $\e_0,\theta\in(0,1)$ such that if for $a,b\in(0,\e_0)$,
\[
(x_n + a)_+ \geq v \geq (x_n - b)_+ \text{ in $B_1$}
\]
then in $B_{1/2}$ either
\[
(x_n + a-\theta c)_+ \geq v \qquad \text{ or } \qquad v \geq (x_n - b +\theta c)_+ \qquad (c=(a+b)/2)
\]
\end{lemma}

Let us briefly recall the ideas from \cite{MR3485136} to prove Lemma \ref{lem:dim_osc0}. Let $P_+(x)=(x_n + a)_+$, $P_-(x)=(x_n - b)_+$, and $P=(x_n + d)_+$ where $d=(a-b)/2$. One has two consider two possible cases, either $u(e_n/2)\geq P(e_n/2)$ or the opposite inequality holds. In the former case one gets to improve the lower bound, i.e. $(x_n -b+\theta c)_+ \geq v$, and in the latter one gets to improve the upper bound by a similar argument. Assuming that $u(e_n/2)\geq P(e_n/2)$, the idea is to apply the classical Harnack inequality to $v-P_-$ around $e_n/2$ and construct a barrier that propagates the improvement beyond $\{x_n=b\}$ thanks to the comparison principle.

In the case that $u$ is only a subsolution of \eqref{eq:vsol} in $B_1$, the barrier argument to improve the upper bound of $u$ still applies. If $u$ is a supersolution of \eqref{eq:vsol} restricted to $B_1^+$, the barrier argument to improve the lower bound can be performed if we assume that the free boundary of the barrier does not reach $\{x_n=0\}$, where $u$ is no longer a supersolution. In this case we can get an improvement proportional to $c=(a+b)/2$ if we assume $b\geq a$.

\begin{corollary}\label{lem:dim_osc1}
There exist $\e_0,\theta\in(0,1)$ such that if for $0<a\leq b<\e_0$,
\[
(x_n + a)_+ \geq u \geq (x_n - b)_+ \text{ in $B_1^+$}
\]
then in $B_{1/2}^+$ either
\[
(x_n + a-\theta c)_+ \geq u \qquad \text{ or } \qquad u \geq (x_n - b +\theta c)_+ \qquad (c=(a+b)/2)
\]
\end{corollary}

By iterating the previous corollary we get the following diminish of oscillation.

\begin{lemma}\label{lem:dim_osc}
There exist $\e_0,\m,\theta\in(0,1)$ such that if for $\e\in(0,\e_0)$ and \eqref{eq:flat_hyp} holds, then in $B_\m^+$ either
\[
x_n \geq u \qquad \text{ or } \qquad u \geq (x_n -(1-\theta)\e)_+.
\]
\end{lemma}

We remark that in the Corollary \ref{lem:dim_osc1} and Lemma \ref{lem:dim_osc} above we do not assume that $0 \in \p' \L$ but only that \eqref{eq:flat_hyp} holds.  If the first alternative of Lemma \ref{lem:dim_osc} holds then $\bar F$ is unconstrained in $B^+_\mu$ and we fall in situation of the interior case as in \cite{MR3485136}. If the second alternative holds then $u$ satisfies a version of \eqref{eq:flat_hyp} in which we replace $B_1^+$ by $B^+_\mu$ and $\e$ by $(1-\theta) \e$.

\begin{proof}
Let $\bar\e_0,\bar\m=1/2,\theta\in(0,1)$ the constants corresponding to Corollary \ref{lem:dim_osc1} and let $\e_0,\m\in(0,1)$ to be fixed later in the proof. As we will be iterating Corollary \ref{lem:dim_osc1} a finite number of times let us actually say that for some $k\in\N$ to be determined $\m=\bar \m^k$.

Let $u_i(x) = \bar\m^{-i}u(\bar\m^i x)$. We have that for any $i=0,1,2,\ldots$
\[
(x_n + \e)_+ \geq (1+\e)x_n \geq u_i \text{ in $B_1^+$}
\]

Let $C_0 = (1-\theta/2)/\bar\m>1$ and $b_i=(2\e/\theta)C_0^i$ for $i\in\{0,1,2,\ldots,(k-1)\}$. Assume by induction that
\[
u_i \geq (x_n-b_i)_+ \text{ in $B_1^+$}
\]
By Corollary \ref{lem:dim_osc1}, we get that in $B_{\bar\m}^+$ either
\[
(x_n+\e-\theta(\e+b_i)/2)_+ \geq u_i \qquad \text{ or } \qquad u_i \geq (x_n-b_i+\theta(\e+b_i)/2)
\]
The first alternative implies $(x_n)_+\geq u$ in $B_\m^+$ and would settle the proof. On the other hand, the second option implies the subsequent step in the induction, $u_{i+1} \geq (x_n-b_{i+1})_+$ in $B_1^+$.

In order to iterate Corollary \ref{lem:dim_osc1} up to $i=k$ we need $b_{k-1}\leq \bar\e_0$ which follows by taking $\e_0 = \theta\bar\e_0/(2C_0^{k-1})$.

We finally fix $k$ sufficiently large such that $1-\theta \geq (2/\theta)(1-\theta/2)^k$. Hence $u_k \geq (x_n -b_k)_+$ in $B_1^+$ implies $u \geq (x_n-(1-\theta)\e)_+$ in $B_\m^+$.
\end{proof}

Next we define the function $w$ in $\overline {\Omega^+}$ as
$$w:= \frac{x_n-u}{\e},$$
and clearly $w \ge 0$ on $\bar F$ and $\|w\|_{L^\infty} \le 1$ if  \eqref{eq:flat_hyp} holds. Lemmas \ref{lem:dim_osc0} and \ref{lem:dim_osc1} provide a diminish of oscillation for $w$ as we restrict to a smaller ball.
By iterating these lemmas (and the standard Harnack inequality at points away from $\{x_n=0\}$) several times we obtain an almost uniform Holder modulus of continuity for $w$ (except for points at smaller and smaller scales). A version of Arzela-Ascoli theorem gives the compactness of a family of $w$'s as $\e ,\delta \to 0$.

Precisely, let us consider a sequence a sequence of solutions $\{u_k\}$ satisfying {\normalfont(\hyperref[eq:flat_hyp]{$F_{\e_k}$})} and {\normalfont(\hyperref[eq:hyp_calpha]{$C_{\e_k, \d_k}$})} with $\e_k,\delta_k \to 0$, and the graphs of the corresponding $w_k$ restricted to the cylinder $\bar B_{1/2} \times \R$,
$$G_k:=\{(x,w_k(x))|\,  x \in \overline {\Omega_k^+} \cap \bar B_{1/2}\}.$$ 

\begin{corollary}\label{cor:cpt}
There exists a subsequence of $G_k$'s which converges (in the Hausdorff distance) to the graph of a Holder continuous function $\bar w \in C(\bar B^+_{1/2})$.
\end{corollary}

Notice that in the previous corollary the domains of definition of $w_k$ vary with $k$, however they converge to $\bar B_1^+$. 

\begin{lemma}
The function $\bar w$ solves the Signorini Problem \eqref{SP} (in the viscosity sense).
\end{lemma}

\begin{proof}
Since $u_k = x_n - \e_k w_k$ and $L_ku_k=0$ we find that $$L_kw_k=\frac{1}{\e_k}L_k x_n=\frac{1}{\e_k} \partial_i(a^{in}_k-\delta^{in}).$$
From {\normalfont(\hyperref[eq:hyp_calpha]{$C_{\e_k,\d_k}$})} we see that as $w_k \to \bar w$, $\delta_k,\e_k \to 0$ we obtain $\D \bar w =0$ in $B_1^+$. 

Since $w_k \ge 0$ on $\bar F_k$ we obtain that $\bar w \ge 0$ on $B_{1}'$. 

It remains to check that on $B_1'$ we satisfy the Signorini condition $\p_n \bar w \leq 0$ in the viscosity sense and we have equality over the positivity set of $\bar w$. Assume that $a+p\cdot x -C|x|^2$ touches $\bar w$ from below at a point $x_0 \in B'_{1/2}$, and assume for simplicity of notation that $x_0=0$. We need to show that $p_n \leq 0$.

Given $\eta>0$ we may assume that the polynomial $P(x) = a+p\cdot x -\eta x_n - C(|x'|^2-nx_n^2)$ touches $w$ strictly from below at $0$ in $B_r^+$ for some small $r$.

Let $\varphi_k = x_n -\e_k P$ and $\widetilde \varphi_k$ such that $L_k \widetilde \varphi_k = \D \varphi_k$ in $B^+_r$, and $\widetilde\varphi_k = \varphi_k$ on $\p B^+_r$. By {\normalfont(\hyperref[eq:hyp_calpha]{$C_{\e_k,\d_k}$})} and Schauder estimates we have $\|\widetilde\varphi_k-\varphi_k\|_{C^{1,\a}(B^+_{r/2})} \leq C \delta_k \e_k$.

By the convergence of $G_k$ to the graph of $w$, we get that for $k$ sufficiently large and some $d_k \in (-\xi,\xi)$, $P+(\varphi_k-\widetilde\varphi_k)/\e_k+d_k$ touches $w_k$ from below at some $x_k \in (\W_k^+ \cup F_k)\cap B_{r/2}$ with $x_k \to 0$. In other words, $\widetilde\varphi_k -\e_kd_k$ touches $u_k$ from above. Given that $L_k\widetilde\varphi_k = -\e_k\D P < 0$ we have that $x_k \in  F_k \cap B_{r/2}$. By the free boundary condition
\[
1- 2 \delta_k \e_k \leq |D_{a_k}\widetilde\varphi_k|^2 \leq |D\varphi_k|^2 + C \delta_k \e_k \leq 1 -2\e_k(p_n-\eta) + C(\delta_k \e_k +\e_k^2),
\]
which implies the desired bound for $p_n$ after we let $k \to \infty$ and then $\eta\to 0$.

A similar argument shows that $\p_n w \geq 0$ over $\{w>0\}\cap B_1'$.
\end{proof}

If we assume that $0 \in \p' \L_k$ then $w_k(0)=0$ and $\bar w(0)=0$. Since $\|\bar w\|_{L^\infty} \le 1$, the optimal $C^{1,1/2}$ regularity for the Signorini problem implies that 
\[
|\bar w(x)| \le C|x|^{3/2},
\]
for some $C$ universal. This implies that given $\beta \in (0,1/2)$, there exists $\mu$ small depending on $\beta$ and the other universal constants such that for all $k$ large
\begin{equation} \label{w}
|w_k| \le \mu^{1+\beta} \quad \mbox{in} \quad \overline{\Omega^+_k} \cap B_{\mu}.
\end{equation}

We have established the following improvement of flatness result.

\begin{lemma}\label{lem:growth}
Given $\b\in(0,1/2)$, there  exist $\e_0,\d,\mu$ depending on $\beta$, and the other universal constants such that if $0\in \p' \L$ and {\normalfont(\hyperref[eq:flat_hyp]{$F_{\e}$})} and {\normalfont(\hyperref[eq:hyp_calpha]{$C_{\e,\d}$})} for some $\e\in(0,\e_0)$ then
\[
x_n+\e\m^{1+\b} \ge u\geq (x_n -\e\m^{1+\b})_+ \text{ in } B_\m^+,
\]
i.e., the rescaling $\tilde u(x):=\mu^{-1}u(\mu x)$ satisfies  {\normalfont(\hyperref[eq:flat_hyp]{$F_{\tilde \e}$})} with $\tilde \e=\e \mu^\beta.$ 
\end{lemma}

The proof of the lemma follows by contradiction and compactness. If the statement fails for a sequence of $u_k$'s, and with corresponding $\e_k$, $\delta_k \to 0$, then we argue as above and find from \eqref{w} that the $u_k$'s do satisfy the conclusion of the lemma for all large $k$. 

We can iterate the lemma above provided that $\beta \le \alpha$ so that hypothesis  \eqref{eq:hyp_calpha} scales accordingly. 
We obtain that $u$ is pointwise $C^{1,\beta}$ at $0 \in \p'\L$ in the domain of definition, i.e.
$$|u-x_n| \le C \e |x|^{1+\beta} \mbox{ in $\overline{\Omega^+}$.}$$

Now it is standard to extend the $C^{1,\beta}$ regularity from $\p' \L$ to the whole domain of definition.

\begin{proposition}\label{thm:almost_opt}
Let $\b\in(0,\min(1/2,\a))$ and assume that $u$ satisfies \eqref{eq:hyp_calpha}, \eqref{eq:flat_hyp} for some $\e\in(0,\e_0)$. Then
\[
\|u\|_{C^{1,\b}( \overline{\Omega^+} \cap B_{1/2})} \leq C.
\]
\end{proposition}

Notice that the estimate above implies that the free boundary $\bar F \in C^{1,\beta}$ as well.

\begin{proof}
It suffices to show that $u$ is pointwise $C^{1,\beta}$ at all points $y \in \overline{\Omega^+} \cap B_{1/2}$. We look at the distance $r$ from $y$ to $\p' \L$, and assume for simplicity that the distance is realized at $0 \in \p' \L$. We assume without loss of generality that $F\cap B_{1/2}\ss \{x_n<|x'|\}$, which follows from Lemma \ref{lem:growth} after a suitable dilation.

By Lemma \ref{lem:growth} $u$ is approximated in a $C^{1,\beta}$ fashion by $x_n$ in balls of radius greater than $r$ centered at $y$.
To check that $u$ is approximated at scales smaller than $r$ we distinguish three cases. 

If $y_n > |y'|/2$ then the desired conclusion follows by interior Schauder estimates.
 
If $y_n \le |y'|/2$ and $B'_{|y'|}(y') \subset \L$ then the conclusion follows by Schauder estimates up to the boundary.

 If $y_n \le |y'|/2$ and $B'_{|y'|}(y') \cap \L=\emptyset$ then $F$ is unconstrained in $B^+_{r/2}(y',0)$. Now the estimates in \cite{MR3485136} apply, or alternatively we could repeat the arguments of Lemma \ref{lem:growth} in the unconstrained setting.
\end{proof}

\begin{remark}\label{r2}
In terms of the function $w$ the estimate we obtained is
$$\e\|w\|_{L^\infty(\Omega^+ \cap B_1)}\leq \e_0 \qquad\Rightarrow\qquad \|w\|_{C^{1,\b}( \overline{\Omega^+} \cap B_{1/2})} \leq C \|w\|_{L^\infty(\Omega^+ \cap B_1)}.$$
\end{remark}

\section{Optimal regularity}\label{sec:opt}

In this section we will establish Theorem \ref{thm:main1}. We assume that $u$ is a solution of \eqref{eq:def} for $$a^{ij}=\d^{ij}, \quad  \W=B_1 \cap \{x_n>g(x')\}, \quad \mbox{and} \quad Z=\{x_n = g(x')\} \cap B_1,$$
 where $g \in C^{1,1/2+\s}(B_1')$ for some small $\s>0$ and
\[
g(0)=0, \qquad D'g(0)=0, \qquad \|g\|_{C^{1,1/2+\s}(B'_1)} \leq 1.
\]
We consider $Q\in C^{0,1}(B_1)$ satisfying
\[
 Q(0)=1, \qquad \|Q-1\|_{C^{0,1}(B_1)}\leq 1,
\]
and assume $0\in\p'\L$, hence $Du(0)=e_n$.

In view of the previous section $u \in C^{1,\b}(\W^+\cap \bar F)$ for some $\b\in(0,1/2)$ that we choose sufficiently close to $1/2$ so that $\b\in(1/2-\s/10,1/2)$.

To establish the $C^{1,1/2}$ regularity of $\bar F$ we follow the strategy from \cite{MR2367025, MR2558329} applied to the function $w$ defined in the previous section as
\[
 w := x_n-u,
\]
and we suppose without loss of generality that
\[
 \|w\|_{C^{1,\b}(\bar \W)} \leq 1.
\]
Since $w(0)=0$, $Dw(0)=0$ we have
\begin{equation}\label{41}
w=O(r^{1+\beta}), \quad |Dw|=O(r^\beta) \quad \mbox{in} \quad \bar \Omega \cap B_r.
\end{equation}
Moreover, the free boundary condition $|Du|=Q$ on $F$ implies $\p_n w=O(r^{2\beta})$ on $F \cap B_r$, or
\begin{equation}\label{42}
\p_\nu w=O(r^{2 \beta}) \quad \mbox{and} \quad w \ge 0 \quad \mbox{on} \quad  F \cap B_r,
\end{equation}
where $\nu$ is the outward normal to $\W^+$. On the remaining part $\L$ of $\bar F \cap B_r$ (where $\bar F$ coincides with $Z$) we have $w=g(x')$ and $|Du| \ge Q$. We easily deduce
\begin{equation}\label{43}
w=O(r^{\frac 32 + \sigma}), \quad Dw \cdot x= O(r^{\frac 32 + \sigma}), \quad \p_\nu w  \ge - C r^{2\beta} \quad \mbox{on} \quad  \L \cap B_r.  
\end{equation}
Combining the inequalities above we find
\begin{equation}\label{44}
w\p_\nu w =O(r^{1+ 3 \beta}+r^{\frac 3 2 + \beta + \sigma}) =O(r^{2+\sigma/2}) \quad \mbox{on} \quad  \bar F \cap B_r.
\end{equation}

The main goal is to use Almgren's monotonicity formula and show that for $r$ is sufficiently small
\begin{align}\label{eq:cthreehalf}
H(r) := \frac{1}{r^{n-1}} \int_{\p B_r(x_0)\cap\W^+}w^2 \leq Cr^{3} ,
\end{align}
from which we can easily deduce that $w=O(r^\frac 32)$.

Below we use the following convention for various average integrals over sets $E \subset \bar B_r$,
\[
\fint_E \, f = \frac{1}{r^d}\int_E f, \quad \quad  \text{ where } d = \dim(E),
\]
hence
\[
H(r) = \fint_{\p B_r\cap\W^+}w^2.
\]

\subsection{Almgren's frequency formula}
If $w$ is a homogeneous function we get that the homogeneity of $w$ can be computed from the \textit{frequency functional}
\[
N(r) = r\frac{d}{dr}\ln \1\fint_{\p B_r}w^2\2^{1/2}.
\]
Almgren's monotonicity formula says that if $w$ is harmonic near the origin, then $N$ is nondecreasing. Moreover, if $N$ remains constant, then $w$ is homogeneous of degree $N$.

Let us compute straightaway the derivative of $H$. In the following $\p_r$ denotes the radial derivative.
\begin{align*}
H'(r) &= 2\fint_{\p B_r\cap \W^+} w\p_r w - \frac{1}{r^2}\fint_{\p B_r\cap \bar F} w^2(x\cdot \nu)\\
&= 2r\fint_{B_r\cap\W^+}|Dw|^2 - 2\fint_{B_r\cap \bar F}w \, \p_\nu w  - \frac{1}{r^2}\fint_{\p B_r\cap \bar F} w^2(x\cdot \nu).
\end{align*}
In order to get an exact formula for the second derivatives we consider the following perturbation of $H$,
$$\widetilde H(r) := H(r) +\int_0^r (E_1(\r)+E_2(\r))\frac{d\r}{\r},$$
with (see \eqref{41}-\eqref{44})
\begin{alignat*}{2}
E_1(r) &:= \frac{1}{r}\fint_{\p B_r\cap \bar F} w^2(x\cdot \nu) = O(r^{2+3\b}),\\
E_2(r) &:= 2r\fint_{B_r\cap \bar F}w\p_\nu w = O(r^{3+\s/2}).
\end{alignat*}
Thus,
\begin{equation}\label{441}
\widetilde H (r)=H(r)+O(r^{3+\s/2}),\qquad \quad \widetilde{H}'(r) = 2r\fint_{B_r\cap\W^+}|Dw|^2,
\end{equation}
and we also have
\[
\frac{n-1}{r}\widetilde{H}'(r)+ \widetilde{H}''(r) = 2\fint_{\p B_r\cap\W^+}|Dw|^2.
\]
By the Rellich's identity $$(n-2)|Dw|^2-2(Dw\cdot x)\D w = \text{div}\1|Dw|^2 x - 2(Dw\cdot x)Dw\2,$$
we obtain
\begin{align*}
(n-2)\int_{B_r\cap\W^+}|Dw|^2 = r & \int_{\p B_r\cap\W^+}\1|Dw|^2 - 2(\p_r w)^2\2 \\
&+ \int_{B_r\cap \bar F} (|Dw|^2(x\cdot\nu)-2(Dw\cdot x)\, \p_\nu w ).
\end{align*}
Using \eqref{41}-\eqref{43} we find that on $\bar F \cap B_r$
$$|Dw|^2(x\cdot\nu)=O(r^{1+3 \beta}), \quad (Dw\cdot x)\, \p_\nu w = O(r^{\frac 32 + \s +\b}),$$
hence
$$ (n-2)\fint_{B_r\cap\W^+}|Dw|^2 =  \fint_{\p B_r\cap\W^+}\1|Dw|^2 - 2(\p_r w)^2\2 + O(r^{1+ \s /2}),$$
which gives
\begin{align}
\label{eq:4} \widetilde{H}''(r) + \frac{1}{r}\widetilde H'(r) &= 4\fint_{\p B_r\cap \W^+}(\p_r w)^2 + O(r^{1+\s/2}).
\end{align}

As in \cite{MR2367025, MR2558329} we consider now a truncated type of frequency
\[
\widetilde N(r) = \frac{r}{2}\frac{d}{dr}\ln \max(\widetilde H(r),r^{3+\s/10}),
\]
and show that it is almost monotone.

\begin{lemma}\label{lem:mono}
\[
\widetilde N'(r) \geq -Cr^{-1+\s/10} \widetilde N(r).
\]
\end{lemma}

First we establish an auxiliary result needed in the proof of Lemma \ref{lem:mono}.

\begin{lemma}\label{lem:aux}
If $\widetilde H(r) \ge r^{3+\s/10}$ then
\begin{align}
\label{eq:3}
\widetilde H'(r) = 2r\fint_{B_r\cap\W^+}|Dw|^2 \ge c_0r^{2+\s/10}.
\end{align}
\end{lemma}

\begin{proof}
 We obtain the lower bound in two steps. Using that $w \geq -r^{3/2+\s}$ over $\bar F\cap B_r$, we get that thanks to the Sobolev and trace inequality
\begin{align}
\label{eq:11}
r^2\fint_{B_r\cap\W^+}|Dw|^2 \geq c\fint_{\p B_r\cap\W^+} [(w+r^{3/2+\s})^-]^2 \geq c\fint_{\p B_r\cap\W^+} (w^-)^2 - Cr^{3+2\s}.
\end{align}

Next we consider the harmonic function $h$ in $\Omega^+ \cap B_r$ such that
\begin{alignat*}{3}
\D h &= 0 \text{ in } B_r \cap \W^+, \qquad &&\p_\nu h = 0 \text { on } B_r\cap \bar F, \qquad &&h=w \text { on } \p B_r \cap \W^+,
\end{alignat*}
and notice that $$\int_{B_r \cap \Omega^+} |Dw|^2 \ge \int_{B_r \cap \Omega^+} |Dh|^2.$$ 
By the maximum principle and using that $\p_\nu w \ge - C r^{2 \beta}$ on $\bar F$, we get that for some $C>0$,
\[
h+Cr^{2 \beta}(x_n-r)\leq w.
\]
Since $w(0)=0$, we find 
\begin{equation}\label{45}
h(0) \le C r^{1+ 2 \beta} \le C r^{3/2 + \s}.
\end{equation} 
Let us assume by contradiction that the conclusion does not hold. Then, by the standard $L^2$ estimates,
$$c_0r^{3+\sigma/10} \ge  r^2\fint_{B_r\cap\W^+}|Dh|^2  \ge c \| h- \bar h\|^2_{L^\infty (B_{r/2}\cap \bar \Omega^+)},$$
where $\bar h$ denotes the average of $h$ over $B_r\cap \W^+$. From \eqref{45} we find $\bar h \le Cc_0 r^{3/2 + \sigma/10},$
and by the Poincar\'e and trace inequality we obtain that
\begin{align}
\label{eq:12}
r^2\fint_{B_r\cap\W^+}|Dh|^2 \geq c\fint_{\p B_r\cap\W^+} (h-\bar h)^2 \geq c\fint_{\p B_r\cap\W^+} (w^+)^2 - Cc_0r^{3+\s/10}.
\end{align}
Now we reach a contradiction by combining \eqref{eq:11} and \eqref{eq:12}, provided that $c_0$ is chosen sufficiently small.
\end{proof}

\begin{corollary}\label{c1}
 If $\widetilde H(r) \ge r^{3+\s/10}$ then
 \[
  \fint_{\p B_r\cap \W^+}w\p_r w \ge c\, r^{2+\s/10}
 \]
\end{corollary}

The corollary follows from Lemma \ref{lem:aux} by noticing that the difference between $\frac 12 \widetilde H'(r)$ and the left-hand side above is $\fint_{\bar F \cap B_r} w \, \p_\nu w=O(r^{2+\s/2})$.

\begin{proof}[Proof of Lemma \ref{lem:mono}]
We focus on the case $\widetilde H(r)>r^{3+\s/2}$ such that Lemma \ref{lem:aux} and its corollary apply. We compute the logarithmic derivative by using \eqref{eq:4},
\begin{align}
\nonumber \frac{\widetilde N'(r)}{\widetilde N(r)} &= \frac{1}{r}+\frac{\widetilde H''}{\widetilde H'}-\frac{\widetilde H'}{\widetilde H}\\
\nonumber &\geq \frac{4\fint_{\p B_r\cap \W^+}(\p_r w)^2 -Cr^{1+\s/2}}{\widetilde H'(r)}-\frac{2\fint_{\p B_r\cap \W^+}w\p_r w+2\fint_{B_r\cap F}w \, \p_\nu w}{\widetilde H(r)}\\
\label{eq:2} &\geq \frac{4\fint_{\p B_r\cap \W^+}(\p_r w)^2}{\widetilde H'(r)}-\frac{2\fint_{\p B_r\cap \W^+}w\p_r w}{\widetilde H(r)}-Cr^{-1+\s/10}.
\end{align}
We use that
$$\widetilde H'=2 \fint_{\p B_r\cap \W^+}w \p_r w + O(r^{2+\s/2}), \quad \quad \widetilde H=\fint_{\p B_r \cap \W^+}w^2+O(r^{3+\sigma/2}),$$
together with Lemma \ref{lem:aux}, Corollary \ref{c1}, and obtain by Cauchy-Schwartz inequality
$$\frac12 \frac{\widetilde N'(r)}{\widetilde N(r)} \ge  \frac{\fint_{\p B_r\cap \W^+}(\p_r w)^2}{\fint_{\p B_r\cap \W^+}w \p_r w} - \frac{\fint_{\p B_r\cap \W^+}w\p_r w}{\fint_{\p B_r \cap \W^+}w^2}-Cr^{-1+\s/10} \ge - C r^{-1+\s/10}.$$
\end{proof}

We have the following consequence of Lemma \ref{lem:aux}.
\begin{corollary}
 $(1+Cr^{\s/10})\widetilde N(r)$ is nondecreasing.
\end{corollary}

\subsection{Blowup}

Our next goal is to show the lower bound 
\begin{equation}\label{N0}
\widetilde N(0^+) \ge \frac 32 .
\end{equation}
We achieve this by blowing up $w$ at the origin so that the blowup limit is a nontrivial homogeneous global solution of the Signorini problem. Then we will obtain the desired bound for $H$ from \eqref{N0} by integrating $(1+Cr^{\s/10})\widetilde N(r)\geq 3/2$.

Let
\[
u_r(x) = \frac{u(rx)}{r} \qquad \e_r = \frac{H(r)^{1/2}}{r} \qquad w_r(x) = \frac{w(rx)}{H(r)^{1/2}} = \frac{x_n-u_r(x)}{\e_r}
\]
with the corresponding domains
\[
\W_r^+ = r^{-1}\W^+, \qquad \qquad \bar F_r = r^{-1} \bar F. 
\]
By construction the $L^2$ norm of $w_r$ over $\p B_1 \cap \W_r^+$ is one.
Also from Remark \ref{r2} we have that given $K\subset\subset B_1$, there exist $\e_0\in(0,1)$ and $C>0$ (depending also on $K$) such that
\begin{equation}\label{48}
\e_r\|w_r\|_{L^\8(B_1\cap \W^+)} \in(0,\e_0)\qquad\Rightarrow\qquad \|w_r\|_{C^{1,\b}( K \cap \bar \W^+)} \leq C\|w_r\|_{L^\8(B_1\cap \W^+)}.
\end{equation}

The following lemma establishes the existence of a homogeneous blowup limit for $w_r$'s.

\begin{lemma}\label{lem:blowup}
If $$\liminf_{r\to 0^+} \frac{\e_r}{r^{1/2+\s/20}} > 1,$$ then there exists a blowup limit $w_0 \in C^{1,\beta}(\bar B_1^+)$ such that for some sequence $r_k\to 0^+$, the graphs of $w_k=w_{r_k}$ converge on compact sets of $B_1 \times \R$  (in the $C^{1,\beta}$ topology) to the graph of $w_0$. Moreover, $w_0$ is a nontrivial solution of the Signorini problem, homogeneous of degree $\widetilde N(0^+)$.
\end{lemma}

First we show that the $L^\8$ norm of $w_r$ can be controlled by the $H^1$ norm in $B_1\cap\W_r^+$.

\begin{lemma}\label{lem:l_infty}
Assume $\widetilde H(r)\geq r^{3+\s/10}$. Given $K\subset\subset B_1$ there exist $C$ (depending on $K$) such that
\[
 \|w_r\|_{L^\8(K\cap\W_r^+)}\leq C(\|w_r\|_{H^1(B_1 \cap \Omega^+_r)}+1).
\]
\end{lemma}

\begin{proof}
Consider $h\geq 0$ such that
\begin{alignat*}{3}
\D h &= 0 \text{ in } B_1 \cap \W_r^+, \qquad &&\p_\nu h = 0 \text { on } B_1\cap F_r, \qquad &&h=w_r^+ \text { on } \p B_1 \cap \W_r^+.
\end{alignat*}
Notice that, $\p_\nu w_r = O(r^{1/2-\s/4})$ over $F_r$, meanwhile $w_r \leq r^{19\s/20}$ over $\L_r$. By the comparison principle we get that $w_r \leq h+1+C(1-x_n)$. Given that $h$ is bounded on $K$ in terms of the $H^1$ norm of $h$ in $B_1 \cap \Omega_r$, which in turn is bounded by the $H^1$ norm of $w^+_r$ in the same domain, we get the desired bound from above.

To obtain the bound from below we consider instead $v\leq 0$ such that
\begin{alignat*}{3}
\D v &= 0 \text{ in } B_1 \cap \W_r^+ \qquad &&v = 0 \text { on } B_1\cap F_r \qquad &&v=(w_r+1)^- \text { on } \p B_1 \cap \W_r^+
\end{alignat*}
Using that $w_r \geq -r^{19\s/20}$ over $F$ we get that $v-1\leq w_r$. Since on $K$, $v$ is bounded by $\|v\|_{H^1}$ which in turn is bounded by $\|(w_r+1)^-\|_{H^1}$ we deduce the desired lower bound.
\end{proof}

\begin{remark}\label{rmk:l_infty}
The same proof applies to $\widetilde w_r(x) = r^{-3/2}w(rx)$ (without the assumption $\widetilde H(r)\geq r^{3+\s/10}$).
\end{remark}

\begin{proof}[Proof of Lemma \ref{lem:blowup}]
Let $r\in(0,r_0)$ with $r_0$ sufficiently small such that $\e_r>r^{1/2+\s/20}$. Given that $\widetilde H(r) = H(r)+O(r^{3+\s/2})$, this implies that, for $r_0$ possibly smaller, we have $\widetilde H(r) \ge r^{3+\s/10}$, and 
\[
 \fint_{B_1\cap \W_r^+} |Dw_r|^2 = \widetilde N(r) \leq C\widetilde N(1).
\]
Because $\|w_r\|_{L^2(\p B_1 \cap \W_r^+)}=1$ we recover that the $H^1$ norm of $w_r$ on $B_1 \cap \W_r^+$ is uniformly bounded. Indeed, one can use that for $(y',y_n)\in (\p B_1)^+$
\[
 \int_{B_1 \cap \W_r^+\cap \{x'=y'\}} w_r^2 \, d x_n \leq C\1w_r^2(y',y_n) + \int_{B_1 \cap \W_r^+\cap \{x'=y'\}} |Dw_r|^2 dx_n \2 .
\]
Then the bound follows after integrating over $y' \in B_1'$.

Consider now an extension to $B_1$ still denoted by $w_r$ and uniformly bounded in $H^1(B_1)$. This means that some sequence $w_k=w_{r_k}$ converges to $w_0$ weakly in $H^1(B_1)$ and strongly in $L^2(\p B_1)$. Moreover $w_0$ is nontrivial because $\|w_0\|_{L^2((\p B_1)^+)}=1$.

The almost optimal regularity Proposition \ref{thm:almost_opt} gives that $\e_r\to0$. From the $L^\8$ bound in Lemma \ref{lem:l_infty} and \eqref{48} we deduce that $w_k$'s are uniformly bounded in $C^{1,\beta}$ in the interior and the convergence to $w_0$ holds in the $C^{1,\b}$ norm on compact sets. Clearly $w_0$ solves the Signorini problem in $B_1^+$, 
and from the convergence in $C^1_{loc}(B_1^+)$ we get that for any $r\in(0,1)$
\[
r^2\frac{\fint_{B_{r}^+}|Dw_0|^2}{\fint_{(\p B_{r})^+}w_0^2} = \lim_{k\to\8} r^2\frac{\fint_{B_{r}\cap\W^+_{r_k}}|Dw_{r_k}|^2}{\fint_{\p B_{r}\cap\W^+_{r_k}}w_{r_k}^2} = \lim_{k\to\8} \widetilde N(rr_k) = \widetilde{N}(0^+).
\]
Given that the standard frequency of $w_0$ is constant we get that it is necessarily a homogeneous function.
\end{proof}

The minimum homogeneity of a nontrivial solution of the Signorini problem is $3/2$. On the other hand if the hypothesis about $\liminf_{r\to 0^+} \e_r/r^{1/2+\s/20}$ is not satisfied, then $\tilde N(0+)=3/2 + \sigma/20$. In conclusion we have established the claim \eqref{N0}.

At this point we are ready to settle our main result.

\begin{proof}[Proof of Theorem \ref{thm:main1}]
After a sufficiently large dilation we can arrange $u$ to satisfy all the hypotheses of this section. Thus, by \eqref{N0}, $(1+Cr^{\s/10})\widetilde N(r) \geq 3/2$ and by integrating,
\[
\widetilde H(r) = H(r) + O(r^{3+\s/2}) \leq Cr^3  \qquad\Rightarrow\qquad H(r) \leq Cr^3.
\]
Now we consider $\widetilde w_r(x) = r^{-3/2}w(rx)$, it is not difficult to check (see \eqref{441}) that its $H^1$ norm is uniformly bounded. Remark \ref{rmk:l_infty} gives the desired modulus of continuity at $0$,
\[
\sup_{r\in(0,r_0)} r^{-3}H(r)\leq C \qquad\Rightarrow\qquad \sup_{r\in(0,r_0)} r^{-3/2} \osc_{\W^+ \cap B_r} w \leq  C,
\]
and we established the pointwise $C^{1,1/2}$ of $u$ at $0 \in \p' \L$. As in the proof of Proposition \ref{thm:almost_opt}, this can be easily extended to $\bar B_{1/2} \cap \bar \Omega^+$ by standard arguments. 
\end{proof}

\bibliographystyle{plain}
\bibliography{mybibliography}

\end{document}